\documentclass[11pt]{article}
\usepackage{amssymb, amsmath, hyperref, array}

\def\R{\mathbb{R}}
\def\N{\mathbb{N}}

\newtheorem{defn}{Definition}
\newtheorem{notn}[defn]{Notation}
\newtheorem{lemma}[defn]{Lemma}
\newtheorem{proposition}[defn]{Proposition}
\newtheorem{theorem}[defn]{Theorem}

\newtheorem{remark}[defn]{Remark}
\newtheorem{example}[defn]{Example}

\newenvironment{proof}[1]{
  \trivlist \item[\hskip \labelsep{\it #1}]}{\hfill\mbox{$\square$}
  \endtrivlist}

\title{A version of Putinar's Positivstellensatz for cylinders}

\date{}

\author{Paula Escorcielo\footnote
{{\scriptsize Partially supported by the Argentinian grants} {\footnotesize UBACYT 20020160100039BA} 
{\scriptsize and} {\footnotesize PIP 11220130100527CO CO\-NI\-CET}.
} 
\qquad \qquad
Daniel Perrucci$^{*}$
\\[3mm]
{\small Departamento de Matem\'atica, FCEN, Universidad de Buenos Aires, Argentina}\\ 
{\small  IMAS, CONICET--UBA, Argentina}
}

\begin{document}

\maketitle

\begin{abstract}
We prove that, under some additional assumption, Putinar's Positivstellensatz
holds on cylinders of type $S \times \R$ with $S = \{\bar x
 \in \R^n \ | \  g_1(\bar x) \ge 0, \dots, g_s(\bar x) \ge 0\}$ 
 such that the quadratic module generated by $g_1, \dots, g_s$ in $\R[X_1, 
 \dots, X_n]$ is archimedean, and we provide a degree bound for the representation 
 of a polynomial $f \in \R[X_1, \dots, X_n, Y]$ which is positive on $S \times \R$
 as an explicit element of the quadratic module generated by $g_1, \dots, g_s$ in $\R[X_1, 
 \dots, X_n, Y]$. We also include an example 
to show that an additional assumption is necessary for Putinar's Positivstellensatz
to hold on cylinders of this type. 
\end{abstract}

\section{Introduction}

Putinar's Positivstellensatz (\cite{Put}) is one of the most celebrated 
results in the theory of sums of squares and certificates of non-negativity. 
This theorem states that
given $g_1, \dots, g_s \in \R[\bar X] = \R[X_1, \dots, X_n]$ such that the quadratic module
$M(g_1, \dots, g_s)$ generated by $g_1, \dots, g_s$ in $\R[\bar X]$ is archimedean, every 
$f \in \R[\bar X]$ positive on 
$$
S = \{\bar x \in \R^n \, | \, g_1(\bar x) \ge 0, \dots, g_s(\bar x) \ge 0\}
$$
belongs to $M(g_1, \dots, g_s)$, which is a certificate of the non-negativity of
$f$.

We explain now the terminology in the preceding paragraph. 
A subset $M \subset \R[\bar X]$ is a quadratic module if it satisfies
\begin{itemize}
\item $1 \in M$,
\item $M + M \subset M$, 
\item $\R[\bar X]^2 M \subset M$ (i.e. $M$ is closed under multiplication by squares).
\end{itemize}
The set of sums of squares $\sum \R[\bar X]^2$ is the smallest quadratic module in $\R[\bar X]$. 
Given $g_1, \dots, g_s \in \R[\bar X]$, the quadratic module generated by these 
polynomials in $\R[\bar X]$ is 
$$
M(g_1, \dots, g_s) = \left\{\sigma_0 + \sigma_1g_1 + \dots + \sigma_sg_s \ | \ 
\sigma_0, \sigma_1, \dots, \sigma_s \in \sum \R[\bar X]^2\right\}
$$
and it is the smallest quadratic module in $\R[\bar X]$ which contains $g_1, \dots, g_s$. 
Every polynomial $f \in M(g_1, \dots, g_s)$ is non-negative on 
the set $S \subset \R^n$ 
but the converse is not true in general (see \cite[Example]{Ste}). 
A quadratic module $M$ in $\R[\bar X]$ is said to be archimedean if there exists 
$N \in \R_{>0}$ such that
$$
N - X_1^2 - \dots - X_n^2 \in M.
$$
If the quadratic module $M(g_1, \dots, g_s)$ is archimedean, 
then the set $S\subset \R^n$ is compact, but again, the converse is not true in 
general (see \cite[Example 4.6]{JacPres}).

Another of the most important results in the theory of sums of squares and certificates
of non-negativity is Schm\"udgen's Positivstellensatz (\cite{Schm}).
This theorem states that
given $g_1, \dots, g_s \in \R[\bar X]$ such that the set $S \subset \R^n$ is compact, 
every 
$f \in \R[\bar X]$ positive on $S$
belongs to the preordering $T(g_1, \dots, g_s)$
generated by $g_1, \dots, g_s$ in $\R[\bar X]$, which is a certificate of the non-negativity of
$f$.

As before,  we explain the terminology we have just used.  
A subset $T \subset \R[\bar X]$ is a preordering if it satisfies
\begin{itemize}
\item $\R[\bar X]^2 \subset T$,
\item $T + T \subset T$, 
\item $TT\subset T$ (i.e. $T$ is closed under multiplication).
\end{itemize}
The set of sums of squares $\sum \R[\bar X]^2$ is the smallest preordering in $\R[\bar X]$.
It is easy to see that $T \subset \R[\bar X]$ is a preordering if and only if it is a 
quadratic module and it is closed under multiplication.
Given $g_1, \dots, g_s \in \R[\bar X]$, the preordering generated by these 
polynomials in $\R[\bar X]$ is 
$$
T(g_1, \dots, g_s) = \Big\{\sum_{I \subset \{1, \dots, s \}}  \sigma_I \prod_{i \in I}g_i
 \ | \ 
\sigma_I \in \sum \R[\bar X]^2 \hbox{ for every } I \subset \{1, \dots, s \} \Big\}
$$
and it is the smallest preordering in $\R[\bar X]$ which contains $g_1, \dots, g_s$.
It is clear that $M(g_1, \dots, g_s)$ is included in $T(g_1, \dots, g_s)$. 
Every polynomial $f \in T(g_1, \dots, g_s)$ is non-negative on 
the set $S \subset \R^n$, 
but the converse is not true in general (again, see \cite[Example]{Ste}). 

Both Schm\"udgen's Positivstellensatz and Putinar's Positivstellensatz
provide a representation of a polynomial on a basic closed semialgebraic
set which makes evident the non-negativity of the polynomial. 
A natural question is if it is possible to bound the degrees of all the different terms  
in these representations. Answers to this question 
have been given by Schweighofer
(\cite[Theorem 3]{Schw}) 
in the case of Schm\"udgen's Positivstellensatz
and by  Nie and Schweighofer (\cite[Theorem 6]{NieSchw})
in the case of Putinar's Positivstellensatz. 
In the particular case where $S$ is the hypercube $[0, 1]^n$, 
improved bounds have been given in \cite{deKLau} and \cite{Mag}.
We include here the precise statement of \cite[Theorem 6]{NieSchw}, but 
we introduce first some useful definition already present in 
\cite{NieSchw}, \cite{Pow}, \cite{PR}, 
\cite{Schw}, etc.

\begin{defn}
For 
$$
f = \sum_{\substack{\alpha \in \N_{0}^n \\ |\alpha| \le d}} \binom{|\alpha|}{\alpha}a_\alpha \bar X^\alpha \in \R[\bar X]
$$ 
we consider the norm of $f$ defined by
$$
\| f \| = \max \{ |a_\alpha| \, | \, \alpha \in \N_{0}^n, |\alpha| \le d \};
$$
where for $\alpha = (\alpha_1, \dots, \alpha_n) \in  \N_{0}^n$, 
$$
|\alpha| = \alpha_1 + \dots + \alpha_n  
\qquad \qquad \hbox{ and } 
\qquad \qquad 
\binom{|\alpha|}{\alpha} = \frac{|\alpha|!}{\alpha_1!\dots \alpha_n!}.
$$
\end{defn}

Note that the definition of this norm is made in such a way that 
for every $d \in \N$, $\|(X_1 + \dots + X_n)^d\| = 1$. 

The precise statement of \cite[Theorem 6]{NieSchw} is the following.

\begin{theorem}[Putinar's Positivstellensatz with degree bound]\label{thm:NieSchw}
Let
$g_1, \dots, g_s \in \R[\bar X]$ 
such that 
$$
\emptyset \ne S  = \{\bar x \in \R^{n} \ | \ g_1(\bar x) \ge 0, \dots, g_s(\bar x) \ge 0\} \subset 
(-1, 1)^n
$$
and such that the quadratic module $M(g_1, \dots, g_s)$ is archimedean. 
There exists a positive constant $c$ such that for every $f \in \R[\bar X]$  
positive on $S$, if
$\deg f = d$ and
$\min\{f(\bar x) \,  | \, \bar x \in S \} = f^* > 0$,   
$f$ can be written as
$$
f = \sigma_0 + \sigma_1g_1 + \dots + \sigma_sg_s \in M(g_1, \dots, g_s)
$$
with $\sigma_0, \sigma_1,  \dots, \sigma_s \in \sum \R[\bar X]^2$ and 
$$
\deg(\sigma_0), \deg(\sigma_1g_1),  \dots, \deg(\sigma_sg_s)
\le 
c \, {\rm e}^{\left(\frac{\| f \|d^2n^d}{f^*}\right)^c}.
$$
\end{theorem}
In the degree bound above, ${\rm e} = 2.718...$ is the base of the natural logarithm. 
Note that the constant $c$ depends on $g_1, \dots, g_s$ but it is independent of $f$. 

The problem of representing 
positive polynomials as sums of squares 
for cylinders with compact cross-section has been studied within the more
general framework of the moment problem in 
\cite{KuhlMars}, \cite{KuhlMarsSchw} and \cite{Pow}.
Under some extra mild assumption, in \cite[Theorem 3]{Pow} Powers obtains an extension of
Schm\"udgen's 
Positivstellensatz to cylinders of
type $S \times F$ with $S \subset \R^n$ a 
compact semialgebraic set and $F \subset \R$ an unbounded closed semialgebraic set. 
The precise extra assumption under consideration is the following.
\begin{defn}\label{def:fullymic} 
Let $f \in \R[\bar X, Y]$,  $m = \deg_Y f$ and $S \subset \R^n$. The polynomial $f$ is \emph{fully} $m$-\emph{ic}
on $S$ if for every $\bar x \in S$, $f(\bar x, Y) \in \R[Y]$ has degree $m$. 
\end{defn}
In other words, the condition of being fully $m$-ic on $S$ is that, when the variable $Y$ is distinguished, the leading 
coefficient (which is a polynomial in $\R[\bar X]$) does not vanish on $S$. 
To obtain \cite[Theorem 3]{Pow}, given $f \in \R[\bar X, Y]$ a positive polynomial on $S \times F$, the idea is to 
consider the variable $Y$ 
as a parameter
and to produce a uniform version of
\cite[Theorem 3]{Schw}, in such a way that all the representations obtained for all the  specializations of
$Y$
can be glued together to obtain the desired representation for $f$.

In this paper we borrow and combine many ideas and techniques from \cite{NieSchw}, 
\cite{Pow} and \cite{Schw} to extend Putinar's Positivstellensatz to cylinders of type $S \times \R$, again under the extra assumption
in Definition \ref{def:fullymic}. 
Before stating our main result, we introduce some definition and notation.

\begin{defn}
For 
$$
f= \sum_{0 \le i \le m} \, \sum_{\substack{\alpha \in \N_{0}^n\\|\alpha| \le d}} \binom{|\alpha|}{\alpha} a_{\alpha,i} \bar{X}^\alpha Y^i \in \R[\bar{X},Y]
$$
we consider another norm of $f$ defined by 
 $$
 \| f \|_{\bullet}= \max\{|a_{\alpha,i}| \, | \, 0 \le  i \le m, \alpha \in \N_0^n, |\alpha| \le d\}.
 $$
\end{defn}

\begin{notn}\label{not:homog}
For 
$$
f = \sum_{0 \le i \le m} f_i(\bar X) Y^i \in \R[\bar{X},Y]
$$ 
with $f_m \ne 0$, we note by 
$$
\bar f = \sum_{0 \le i \le m} f_i(\bar X) Y^iZ^{m-i} \in \R[\bar{X},Y, Z]
$$ 
its homogenization with respect to the variable $Y$. 
\end{notn}

Let $\emptyset \ne S \subset \R^n$ be a compact set and let $f \in \R[\bar X, Y]$. 
If $f$ is fully $m$-ic on $S$ and $f > 0$ on $S \times \R$, it is clear
that $m$ is even and $f_m >0$ on $S$.  It can also be easily seen that if 
we take 
$$
C = \{(y, z) \in \R^2, y^2 + z^2 = 1\}
$$ 
then $\bar f > 0$ on the compact set $S \times C$.

\begin{notn}
For $g_1, \dots, g_s \in \R[\bar X]$, 
we denote
$$
M_{\R[\bar X, Y]}(g_1, \dots, g_s) = \left\{ \sigma_0 + \sigma_1g_1 + \dots + \sigma_sg_s \, | \, 
\sigma_0, \sigma_1, \dots, \sigma_s \in \sum \R[\bar X, Y]^2\right\}$$
the quadratic module generated by $g_1, \dots, g_s$ in $\R[\bar X, Y]$. 
\end{notn}
Note that we keep the notation $M(g_1, \dots, g_s)$ for 
the quadratic module generated by $g_1, \dots, g_s$ in $\R[\bar X]$. 
We state now our main theorem.

\begin{theorem}\label{thm:main}
Let
$g_1, \dots, g_s \in \R[\bar X]$ 
such that 
$$
\emptyset \ne S  = \{\bar x \in \R^{n} \ | \ g_1(\bar x) \ge 0, \dots, g_s(\bar x) \ge 0\} \subset (-1, 1)^n
$$
and such that the quadratic module $M(g_1, \dots, g_s)$ is archimedean. 
There exists a positive constant $c$ such that for every $f \in \R[\bar X, Y]$ positive on $S \times \R$, 
if $\deg_{\bar X} f = d$, $\deg_{Y} f = m$ with $f$ fully $m$-ic on $S$ and 
$$
\min\{\bar f(\bar x, y, z) \,  | \, \bar x \in S, (y, z) \in C \} = f^{\bullet} >0, 
$$ 
$f$ can be written as
$$
f = \sigma_0 + \sigma_1g_1 + \dots + \sigma_sg_s \in M_{\R[\bar X, Y]}(g_1, \dots, g_s)
$$
with $\sigma_0, \sigma_1,  \dots, \sigma_s \in \sum \R[\bar X, Y]^2$ and 
$$
\deg(\sigma_0), \deg(\sigma_1g_1),  \dots, \deg(\sigma_sg_s) 
\le c (m+1)2^{\frac{m}{2}}  {\rm e}^{ \left( \frac{\| f \|_{\bullet}(m+1)d^2(3n)^d}{f^{\bullet}} \right)^{c} }.
$$
\end{theorem}
As in Theorem \ref{thm:NieSchw},  the constant $c$ depends on $g_1, \dots, g_s$ but it is independent of $f$. 
Note that, when $m = 0$, this is to say, $f \in \R[\bar X]$, the bound in Theorem \ref{thm:main} 
is of similar type to the bound in
Theorem \ref{thm:NieSchw}. 
Actually, in Remark \ref{obs:3^d} we see that if $n \ge 2$, the factor $3^d$ in the exponent can be hidden in the constant $c$ and therefore the bound in 
Theorem \ref{thm:main} is of the same type to the bound in Theorem \ref{thm:NieSchw}.

Theorem \ref{thm:main} is basically Putinar's Positivstellensatz 
under the additional assumption
that $f$ is fully $m$-ic on $S$. Next example, which is a variation 
of \cite[Example]{Ste} shows that 
either this one or some other additional assumption is indeed necessary. 

\begin{example}
Take $g_1 = (1 - X^2)^3 \in \R[X]$,  then $S = [-1, 1] \subset \R$ and $M(g_1)$ is archimedean since
$$
\frac43 - X^2 = \frac43X^2\Big(X^2 - \frac32\Big)^2 + \frac43\Big(1-X^2\Big)^3.
$$
Now take $f(X, Y) = (1 - X^2)Y^2 + 1 \in \R[X, Y]$. It is clear that $f > 0$ in 
$S \times \R$ but $f$ is not fully $2$-ic on $S$.
If $f \in M_{\R[X, Y]}(g_1)$, we have an identity
\begin{equation}\label{eq:example}
(1 - X^2)Y^2 + 1 = \sum_{1 \le j \le s}\Big( \sum_{0 \le i \le m'}p_{ji}(X)Y^i \Big)^2 + 
\sum_{1 \le j \le s}\Big( \sum_{0 \le i \le m'}q_{ji}(X)Y^i \Big)^2
\big(1 - X^2\big)^3
\end{equation}
with 
at least one of 
$p_{1m'}, \dots, p_{sm'}, q_{1m'}, \dots, q_{sm'}$ not identically zero. 
Looking at the degree in $Y$ at both sides of (\ref{eq:example}), we have $m' \ge 1$. 

If $m' \ge 2$, looking at the terms of degree $2m'$ in $Y$ at both sides of 
(\ref{eq:example}) we have
$$
0 = \sum_{1 \le j \le s}p_{jm'}(X)^2  + 
\sum_{1 \le j \le s} q_{jm'}(X)^2\big(1 - X^2\big)^3
$$
but this is impossible since the polynomial on the right hand side is positive in $[-1, 1]$
with the only possible exception of a finite number of points. 
Indeed, any point in $(-1,1)$ such that the polynomial on the right 
hand side vanishes at, should be a common root of 
$p_{1m'}, \dots, p_{sm'}, q_{1m'}, \dots, q_{sm'}$; but at least one of these 
polynomials is not identically zero and therefore has a finite number of roots.

If $m' =1$, looking at the terms of degree $2$ in $Y$ at both sides of 
(\ref{eq:example}) we have
$$
1-X^2 = \sum_{1 \le j \le s_1}p_{j1}(X)^2  + 
\sum_{1 \le j \le s_2} q_{j1}(X)^2\big(1 - X^2\big)^3 \in M(g_1)
$$
and this is impossible since it is 
exactly the well-known example from \cite[Example]{Ste}. 
\end{example}

As in \cite[Theorem 3]{Pow}, the general idea to prove Theorem \ref{thm:main} is, 
given $f \in \R[\bar X, Y]$ a positive polynomial on $S \times \R$, to consider the variable $Y$ as a parameter and to produce, 
this time, a uniform version 
of Theorem \ref{thm:NieSchw} (\cite[Theorem 6]{NieSchw}), in such a way that all the representations 
obtained for all the specializations of $Y$
can be glued together to obtain the desired representation for $f$. 

Actually, the proof of \cite[Theorem 6]{NieSchw} uses \cite[Theorem 3]{Schw}, and the proof of 
\cite[Theorem 3]{Schw} uses the bound for P\'olya's Theorem from \cite[Theorem 1]{PR}.
For us, in order to succeed to prove Theorem \ref{thm:main} following the described strategy, we need to reorganize
these ideas in a way that we use directly \cite[Theorem 1]{PR} without going through \cite[Theorem 3]{Schw} as a packaged theorem
(even though we use ideas from its proof). 

Since P\'olya's Theorem plays such a significant role, we first prove 
the following auxiliary proposition, which is a 
a version of Theorem \ref{thm:main} under the extra 
assumption that the set $S$ is included in the interior of a convenient simplex. 
Then, Theorem \ref{thm:main} is obtained by simply composing with a linear change of variables.

\begin{notn} For $n \in \N$, we denote by $\widetilde{\Delta}_n$ the simplex 
$$
\widetilde{\Delta}_n= \Big\{ \bar x \in \R^n \ | \  \displaystyle \sum_{1 \le i \le n} x_i \leq 1 \hbox{ and } x_i \geq 0 
\hbox{ for } 1 \le i \le n \Big\}.
$$
\end{notn}

\begin{proposition}\label{prop:main}
Let $g_1, \dots, g_s \in \R[\bar{X}]$ such that 
$$
\emptyset \ne S = \{\bar x \in \R^{n} \ | \ g_1(\bar x) \ge 0, \dots, g_s(\bar x) \ge 0\} \subset {\widetilde{\Delta}_n}^{\circ}
$$
and such that the quadratic module $M(g_1, \dots, g_s)$ is archimedean. 
There exists a positive constant $c$ such that for every $f \in \R[\bar X, Y]$ positive on $S \times \R$, 
if $\deg_{\bar X} f = d$, $\deg_{Y} f = m$ with $f$ fully $m$-ic on $S$ and 
$$
\min\{\bar f(\bar x, y, z) \,  | \, \bar x \in S, (y, z) \in C \} = f^{\bullet} >0, 
$$ 
$f$ can be written as
$$
f = \sigma_0 + \sigma_1g_1 + \dots + \sigma_sg_s \in M_{\R[\bar X, Y]}(g_1, \dots, g_s)
$$
with $\sigma_0, \sigma_1,  \dots, \sigma_s \in \sum \R[\bar X, Y]^2$ and 
$$
\deg(\sigma_0), \deg(\sigma_1g_1),  \dots, \deg(\sigma_sg_s)
\le 
c (m+1)2^{\frac{m}{2}}  {\rm e}^{\left( \frac{\|  f \|_{\bullet}(m+1)d^2}{f^{\bullet}} \right)^{c}}.
$$
\end{proposition}

A nice fact about the bound in Proposition \ref{prop:main} is 
that it is singly exponential in $d$, meanwhile
the bound in Theorem \ref{thm:main} is doubly
exponential in $d$. This could be of independent interest even in the 
case $m = 0$, this is to say, $f \in \R[X]$ positive on $S$.

\section{Proof of the main result}

As said in the introduction, to prove Theorem \ref{thm:main} we borrow and combine 
many ideas and
techniques from \cite{NieSchw}, \cite{Pow} and \cite{Schw}. 
We start this section by quoting some results from these papers.
We will also use many other general ideas from these sources which is not 
possible to quote independently but we want to give them credit for.

The following two auxiliary results come from \cite[Remark 12]{NieSchw}
and \cite[Proposition 14]{NieSchw}.

\begin{remark} \label{obs:cota_g}
For every $k \in \N$ and $t \in [0,1]$, 
 $$
 t \cdot (t-1)^{2k} \leq \frac{1}{2k+1}.
 $$
\end{remark}

\begin{proposition}\label{lema:norma}
Let $g_1, \dots, g_s \in \R[\bar X] \setminus\{0\}$. Then
$$
\|g_1\dots g_s\| \leq (\deg g_1 +1) \dots (\deg g_s +1) \|g_1 \| \dots \|g_s \|.
$$
If in addition $g_1, \dots, g_s$ are homogeneous, then 
$$
\|g_1\dots g_s\| \leq \|g_1 \| \dots \|g_s \|.
$$
\end{proposition}

Next remark is present in the proof of \cite[Lemma 13]{NieSchw}
and is similar to a remark made in the proof of 
\cite[Lemma 9]{Schw}.

\begin{remark}\label{obs:loja}
Let $g_1, \dots, g_s \in \R[\bar{X}]$ such that 
$$
\emptyset \ne S  = \{\bar x \in \R^{n} \ | \ g_1(\bar x) \ge 0, \dots, g_s(\bar x) \ge 0\} \subset {\widetilde{\Delta}_n}^{\circ}.
$$
By \L{}ojasiewicz inequality (see \cite[Corollary  2.6.7]{BCR}), there exist 
$c_1,c_2>0$ such that for every $\bar x \in \widetilde{\Delta}_n$,
$$
{\rm dist}(\bar{x},S)^{c_1} \leq -c_2 \min\{g_1(\bar{x}),\cdots,g_s(\bar{x}),0 \}.
$$
In particular, for $x \in {\widetilde{\Delta}_n} \setminus S$,  there exists $i_0$ with $1  \le i_0 \le s$ 
such that $g_{i_0}(\bar{x}) < 0$ and
$$
{\rm dist}(\bar{x},S)^{c_1} \leq -c_2 g_{i_0}(\bar{x}).
$$
\end{remark}

The following remark will be useful in the proof of our main result
and the idea of using a Putinar representation for a finite number of 
fixed polynomials comes from the proof of \cite[Theorem 3]{Schw}.

\begin{remark}\label{obs:exis}
Let $g_1, \dots, g_s \in \R[\bar{X}]$ such that 
$$
\emptyset \ne S = \{\bar x \in \R^{n} \ | \ g_1(\bar x) \ge 0, \dots, g_s(\bar x) \ge 0\} \subset {\widetilde{\Delta}_n}^{\circ}
$$
and $M(g_1, \dots, g_s)$ is archimedean.
Since for each $v = (v_0, \bar v) \in \{0,1\}^{n+1}$
$$
(1-X_1- \cdots -X_n)^{v_0} \bar{ X}^{\bar v} > 0 \hbox{ in } S,
$$
by Putinar's Positivstellensatz (\cite{Put}), there exist $\sigma_{v 0}, \sigma_{v 1}, \dots,  \sigma_{v s} \in  
\sum \R[\bar X]^2$ such that 
$$
(1-X_1- \cdots -X_n)^{v_0} \bar{X}^{\bar v} = \sigma_{v 0} + \sigma_{v 1}g_1 + \dots + \sigma_{v s}g_s.
$$
\end{remark}

We need to extend the definition of $\|{\ }\|_\bullet$ to polynomials
homogeneous in $(Y, Z)$ as follows. 
\begin{defn}
For 
$$
h= \sum_{0 \le i \le m} \, 
\sum_{\substack{\alpha \in \N_{0}^n \\ |\alpha| \le d}} \binom{|\alpha|}{\alpha} a_{\alpha,i} \bar{X}^\alpha Y^iZ^{m-i} \in \R[\bar{X},Y, Z]
$$
we consider the new norm defined also for $h$  by 
$$
 \| h \|_{\bullet}= \max\{|a_{\alpha,i}| \, | \, 0 \le  i \le m, \alpha \in \N_0^n, |\alpha| \le d\}.
$$
\end{defn}
Note that  $\|\bar f \| = \| f \|$ for every $f \in \R[\bar X, Y]$.

Next auxiliary lemma will be useful to prove the degree bound from 
Theorem \ref{thm:main}. We use 
the notation $C =  \{(y, z) \in \R^2 \ | \ y^2 + z^2 = 1\}$ which we introduced before
and we keep for the rest of the paper. 

\begin{lemma}\label{lem:norma_div_min}
Let $f \in \R[\bar{X},Y]$ such that $\deg_{\bar{X}}f=d$ and 
$\deg_{Y} f=m$. 
For every $\bar x \in \widetilde \Delta_n$, and $(y, z) \in C$,
$$
|\bar f(\bar x, y, z)| \le \| f\|_{\bullet} (m+1)(d+1).
$$
\end{lemma}
\begin{proof}{Proof:}
Suppose
$$
f= \sum_{0 \le i \le m} \, 
\sum_{\substack{\alpha \in \N_{0}^n \\ |\alpha| \le d}} 
\binom{|\alpha|}{\alpha} a_{\alpha,i} \bar{X}^\alpha Y^i.
$$
Then
\begin{equation*}
 \begin{split}
|\bar f(\bar x, y, z)| & \leq
\sum_{0\leq i \leq m} \sum_{\substack{ \alpha \in \N_0^n \\ |\alpha| \leq d}} \binom{|\alpha|}{\alpha} |a_{\alpha,i}| \bar{x}^\alpha |y|^i |z|^{m-i} \\
& \leq \| f \|_{\bullet}\sum_{0\leq i \leq m} \sum_{\substack{ \alpha \in \N_0^n \\ |\alpha| \leq d}} \binom{|\alpha|}{\alpha}\bar{x}^\alpha  \\
& = \| f \|_{\bullet}\sum_{0\leq i \leq m} \sum_{0\leq j\leq d} \sum_{\substack{ \alpha \in \N_0^n \\ |\alpha| =j}} \binom{|\alpha|}{\alpha}\bar{x}^\alpha \\
& = \| f \|_{\bullet}\sum_{0\leq i \leq m} \sum_{0\leq j\leq d} (x_{1}+ \cdots +x_{n})^j \\
&\leq  \| f \|_{\bullet}\sum_{0\leq i \leq m} \sum_{0\leq j\leq d} 1 =\|f\|_{\bullet}(m+1)(d+1).
 \end{split}
\end{equation*}
\end{proof}

The following lemma is an adaptation from \cite[Lemma 11]{NieSchw}.

\begin{lemma}\label{lema_desig}
Let $f \in \R[\bar{X},Y]$ such that $\deg_{\bar{X}}f=d$ and 
$\deg_{Y} f=m$. For every 
$\bar{x}_1,\bar{x}_2 \in \widetilde{\Delta}_n$ and $(y,z) \in C$, 
 $$
 |\bar f(\bar{x}_1,y,z)-\bar f(\bar{x}_2,y,z)| \leq \frac{1}{2} \sqrt{n} \| f \|_{\bullet}(m+1)d(d+1) \| \bar{x}_1-\bar{x}_2\|.
 $$
\end{lemma}

We include also a technical lemma similar to \cite[Lemma 15]{NieSchw}. 

\begin{lemma}\label{lem:ctes}
Given $(c_1,c_2,c_3,c_4, c_5, c_6) \in \R_{\ge 0}^6$, there exists a positive constant
$c$ such that for every $r \in \R_{\ge 0}$,
$$
c_1r^{c_2}
\le c \,  {\rm e}^{r^{c}}
\qquad
\hbox{ and }
\qquad
c_3r^{c_4}
{\rm e}^{c_5r^{c_6}} \le c \, {\rm e}^{r^{c}}.
$$
\end{lemma}

Before proving Proposition \ref{prop:main} we include some more notation.

\begin{notn} For $n \in \N$, we denote, as usual, by $\Delta_n$ the simplex
$$ 
\Delta_n= \Big\{(x_0, \bar x)\in \R^{n+1} \ | \ \displaystyle \sum_{0 \le i \le n} x_i=1 \hbox{ and } x_i \geq 0 
\hbox{ for }  0 \le i \le n \Big\}.
$$
\end{notn}

We are ready to prove Proposition \ref{prop:main}.

\begin{proof}{Proof of Proposition \ref{prop:main}:}
Without loss of generality we suppose $\deg g_i \ge 1$ and  $|g_i|\leq 1$ in $\widetilde \Delta_n$ for $1 \le i \le s$. 

To prove the result, we take such a polynomial $f \in \R[\bar X, Y]$ and we
need to show that we can find a constant $c$ which works independently from $f$. 
If $d = 0$ then $f \in \R[Y]$ is positive on $\R$ and it is well-known that $f \in
\sum \R[Y]^2$. 
Moreover, we can write $f$ as a sum of squares with 
the degree of each square bounded by $m$ (see \cite[Proposition 1.2.1]{Marshall_book}),
and then the degree bound simply holds for any constant $c \ge 1$.
So from now
we suppose $d \ge 1$ and if the final constant $c$ we find turns out to be less than $1$, we just
replace it by the result of applying Lemma \ref{lem:ctes} to 
the $6$-uple $(1, 0, c, 0, 1, c)$.  The new constant $c$ in particular satisfies for every $r \in \R_{\ge0}$,
$$
1 \le c \,  {\rm e}^{r^{c}}
$$
and taking $r = 0$ we have $c \ge 1$. 

We prove first that there exist 
$\lambda \in \R_{>0}$ and $k \in \N_0$ 
such that
$$
h= \bar{f}- \lambda\big( Y^2 + Z^2 \big)^{\frac{m}2} 
\sum_{1 \leq i \leq s} g_i \cdot (g_i-1)^{2k}  \in \R[\bar{X},Y,Z]
$$
satisfies $h \ge \frac12 f^{\bullet}$ in $\widetilde{\Delta}_n \times C$.

For each $(y,z) \in C$ we consider 
$$
A_{y,z}=\left\{\bar{x} \in \widetilde{\Delta}_n \ | \ \bar{f}(\bar x, y, z) \leq \frac{3}{4}f^{\bullet}\right\}.
$$
Note that $A_{y, z} \cap S = \emptyset$. 

To exhibit sufficient conditions for $\lambda$ and $k$, we consider separately the cases 
$\bar{x} \in \widetilde{\Delta}_n \setminus A_{y,z}$ and 
$\bar{x} \in A_{y,z}$.

If $\bar{x} \in \widetilde{\Delta}_n \setminus A_{y,z}$, by Remark \ref{obs:cota_g},
\begin{equation*}
\begin{split}
 h(\bar{x},y,z) & =\bar{f}(\bar{x},y,z)- 
 \lambda \big( y^{2} + z^{2} \big)^{\frac{m}2} 
  \sum_{1 \leq i \leq s} g_i(\bar{x}) \cdot (g_i(\bar{x})-1)^{2k} \\
  & \ge \bar{f}(\bar{x},y,z)- 
 \lambda  
  \sum_{1 \leq i \leq s} |g_i(\bar{x})| \cdot (|g_i(\bar{x})|-1)^{2k} \\
 &> \frac{3}{4}f^{\bullet} - 
  \frac{\lambda s}{2k+1}. 
 \end{split}
\end{equation*}
Therefore the condition $h(\bar{x},y,z) \ge \frac 12 f^{\bullet}$ is ensured if
\begin{equation}\label{eq:cond_k}
2k+1\ge \frac{4 \lambda  s}{f^{\bullet}}.
\end{equation}

If $\bar{x} \in A_{y,z}$, for any $\bar{x}_0 \in S$, by Lemma \ref{lema_desig} we have
$$
\frac{f^{\bullet}}{4} \leq \bar{f}(\bar{x}_0, y, z)-\bar{f}(\bar{x}, y, z) \leq \frac{1}{2} \sqrt{n} \| f \|_{\bullet}(m+1)d(d+1) \| \bar{x}_0-\bar{x}\|,
$$
then
$$
\frac{f^{\bullet}}{2\sqrt{n} \| f \|_{\bullet}(m+1)d(d+1)} 
\leq \| \bar{x}_0-\bar{x}\| 
$$
and therefore
\begin{equation}\label{eq:dist}
\frac{f^{\bullet}}{2\sqrt{n} \| f \|_{\bullet}(m+1)d(d+1)} \leq \mbox{dist}(\bar{x},S).
\end{equation}

By Remark \ref{obs:loja}, there exist $c_1, c_2 > 0$ and 
$1 \le i_0 \le s$ such that $g_{i_0}(\bar x) < 0$ and
\begin{equation}\label{eq:loja}
\mbox{dist}(\bar x, S)^{c_1} \le -c_2g_{i_0}(\bar x).
\end{equation}
By (\ref{eq:dist}) and (\ref{eq:loja}) we have
\begin{equation}\label{desig_6}
 g_{i_0}(\bar{x})\leq - \delta.
\end{equation}
with 
$$
\delta= \frac{1}{c_2} \left( \frac{f^{\bullet}}{2\sqrt{n}\| f \|_{\bullet}(m+1)d(d+1)} \right)^{c_1} > 0.
$$

On the other hand, defining $f^{\bullet}_{y,z}= 
\min\{\bar f(\bar x, y, z) \,  | \, \bar x \in S \}$,
again by Lemma \ref{lema_desig} we have that
\begin{equation}\label{desig_7}
|\bar{f}(\bar{x},y,z)-f^{\bullet}_{y,z}| \leq \frac{1}{2} \sqrt{n}\| f \|_{\bullet}(m+1)d(d+1) 
\mbox{diam} (\widetilde{\Delta}_n)
= \frac{1}{\sqrt{2}} \sqrt{n}\| f \|_{\bullet}(m+1)d(d+1).
\end{equation}

Then, by Remark \ref{obs:cota_g} and using (\ref{desig_6}) and (\ref{desig_7}) we have 
\begin{equation*}
 \begin{split}
 h(\bar{x},y,z)& \geq \bar{f}(\bar{x},y,z)
 -\lambda g_{i_0}(\bar{x})(g_{i_0}(\bar{x})-1)^{2k} - 
  \frac{\lambda(s-1)}{2k+1} \\
 & \geq \bar{f}(\bar{x},y,z) -f^{\bullet}_{y,z} +f^{\bullet}_{y,z} + 
 \lambda  \delta - \frac{\lambda(s-1)}{2k+1}  \\
 & \geq - \frac{1}{\sqrt{2}} \sqrt{n}\|f \|_{\bullet}(m+1)d(d+1) 
 + f^{\bullet} + \lambda \delta - \frac{\lambda(s-1)}{2k+1}.
 \end{split}
\end{equation*}
Finally, the condition $h(\bar{x},y,z) \geq \frac12f^{\bullet}$ is ensured if
\begin{equation}\label{eq:cond_lambda}
\lambda \ge \frac{ \sqrt{n}\|f \|_{\bullet}(m+1)d(d+1)}{\sqrt{2}\delta}  = 
\frac{ c_2 2^{c_1} (\sqrt{n}\| f \|_{\bullet}(m+1)d(d+1))^{c_1+1}}{\sqrt{2}{f^{\bullet}}^{c_1}}
\end{equation}
and 
\begin{equation}\label{eq:aux_inutil}
2k+1 \ge \frac{2\lambda(s-1)}{f^{\bullet}}. 
\end{equation}
Since (\ref{eq:cond_k}) implies (\ref{eq:aux_inutil}), it is enough for $\lambda$ and $k$ to satisfy (\ref{eq:cond_k}) and (\ref{eq:cond_lambda}). 
So for the rest of the proof we take
$$
\lambda = \frac{ c_2 2^{c_1} (\sqrt{n}\| f \|_{\bullet}(m+1)d(d+1))^{c_1+1}}{\sqrt{2}{f^{\bullet}}^{c_1}} 
= c_3\frac{(\| f \|_{\bullet}(m+1)d(d+1))^{c_1+1}}{{f^{\bullet}}^{c_1}}
> 0$$ 
with $c_3 = \frac{c_22^{c_1}\sqrt{n}^{c_1+1}}{\sqrt{2}}$
and
$$
k = \left\lceil \frac12 \left( \frac{4\lambda s}{f^{\bullet}} -1\right) \right\rceil \in \N_0.
$$
In this way, 
\begin{equation}\label{eq:cota_k}
\begin{split}
 k & \le   \frac12 \left( \frac{4\lambda s}{f^{\bullet}} -1\right) +1  \\
   & = 2c_3s \left( \frac{\| f \|_{\bullet}(m+1)d(d+1)}{f^{\bullet}} \right)^{c_1+1}+\frac12 \\
   & \leq c_4 \left( \frac{\| f \|_{\bullet}(m+1)d(d+1)}{f^{\bullet}} \right)^{c_1+1}
 \end{split}
\end{equation}
with $c_4 = 2c_3s + 1$. Here (and also several times after here) we use Lemma \ref{lem:norma_div_min} to ensure
$$\frac{\| f \|_{\bullet}(m+1)(d+1)}{f^{\bullet}} \ge 1.$$

Also, if we define $\ell = \deg_{\bar X} h$, we have
\begin{equation}\label{cota_e}
 \begin{split}
\ell & \leq \max \{d, (2k+1)\max_{1 \le i \le s} \deg g_i  \} \\ 
& \leq \max \left\{
d,  
\left(2c_4\left( \frac{\| f \|_{\bullet}(m+1)d(d+1)}{f^{\bullet}} \right)^{c_1+1} +1 \right) 
\max_{1 \le i \le s} \deg g_i
\right\} \\
 &
\le
c_5 \left( \frac{\| f\|_{\bullet}(m+1)d(d+1)}{f^{\bullet}} \right)^{c_1+1}
 \end{split}
\end{equation}
with $c_5 = (2c_4 + 1)\displaystyle{\max_{1 \le i \le s} \deg g_i}$.

On the other hand, using Proposition \ref{lema:norma} and (\ref{eq:cota_k}), 
\begin{equation}\label{cota_h}
 \begin{split}
\|h\|_{\bullet} 
&
\le  {\| f \|_{\bullet}} + \lambda s  2^{\frac{m}{2}}   
\max_{1 \le i \le s} \{(\deg g_i +1) (\|g_i\| +1)\} ^{2k+1}   \\[3mm]
& =  
\| f \|_{\bullet} +  
c_3s2^{\frac{m}2}\frac{ (\| f \|_{\bullet}(m+1)d(d+1))^{c_1+1}}{{f^{\bullet}}^{c_1}}
\max_{1 \le i \le s} \{(\deg g_i +1) (\|g_i\| +1)\}^{2k+1} \\[3mm]
& \le 
(c_3s+1)\max_{1 \le i \le s} \{(\deg g_i +1) (\|g_i\| +1)\} \ \cdot 
\\
& \ \ \ \cdot \
2^{\frac{m}2}\frac{ (\| f \|_{\bullet}(m+1)d(d+1))^{c_1+1}}{{f^{\bullet}}^{c_1}}
\max_{1 \le i \le s} \{(\deg g_i +1) (\|g_i\| +1)\}^{2c_4 \left( \frac{\| f \|_{\bullet}(m+1)d(d+1)}{f^{\bullet}} \right)^{c_1+1}}
\\[3mm]
& =
c_62^{\frac{m}2}
\frac{ (\| f \|_{\bullet}(m+1)d(d+1))^{c_1+1}}{{f^{\bullet}}^{c_1}}
{\rm e}^{c_7 \left( \frac{\| f \|_{\bullet}(m+1)d(d+1)}{f^{\bullet}} \right)^{c_1+1} }
\end{split}
\end{equation}
with 
$c_6 = (c_3s + 1)\displaystyle{\max_{1 \le i \le s} \{(\deg g_i +1) (\|g_i\| +1)\}}$ 
and 
$c_7 = \log \left(\displaystyle{\max_{1 \le i \le s} \{(\deg g_i +1) (\|g_i\| +1)\}^{2c_4}}\right)$.

So far we have found $\lambda$ and $k$ such that 
that $h \ge \frac12 f^{\bullet}$ in $\widetilde \Delta_n \times C$, together with bounds for 
$k, \ell = \deg_{\bar X} h$ and $\|h\|_\bullet$. 
Now, 
we introduce a new variable $X_0$ with the aim of  
homogenize with respect to the variables $\bar X$ and be able to 
use P\'olya's Theorem.
Let 
$$
h=  \sum_{0 \leq i \leq m}  \sum_{0 \leq j \leq \ell} h_{ij}(\bar{X}) Y^i Z^{m-i} 
$$
with $h_{ij} \in \R[\bar X]$ equal to zero or homogeneous of degree  $j$ for $0 \le i\le  m$ and 
$0 \le j \le \ell$. 
We define
$$
H= \sum_{0 \leq i \leq m}  \sum_{0 \leq j \leq \ell} h_{ij}(\bar{X})
(X_0+X_1 \cdots +X_n)^{\ell-j}  Y^i Z^{m-i} \in \R[X_0, \bar{X},Y,Z]
$$
which is bihomogeneous in $(X_0, \bar X)$ and $(Y, Z)$
(i.e. homogeneous in the variables $(X_0, \bar X)$ and $(Y, Z)$ separately). 

Since $H(x_0, \bar x, y, z) = h(\bar x, y, z)$ for every $(x_0, \bar x, y, z) \in 
\Delta_n \times C$, it is clear that $H\ge\frac12f^{\bullet}$ in $\Delta_n \times C$.

On the other hand, for each $(y, z) \in C$, we consider $H(X_0, X, y, z) \in 
\R[X_0, \bar X]$. Using Proposition \ref{lema:norma} we have 
$$
\begin{array}{rcl}
\|H(X_0, X,y,z) \| & 
\leq  & \displaystyle{\sum_{0 \leq i \leq m}   \sum_{0 \leq j \leq \ell} \|h_{ij}(\bar X)(X_0+\cdots+X_n)^{\ell-j}y^iz^{m-i} \|   }
\\[6mm]
& \leq & \displaystyle{ \sum_{0 \leq i \leq m}  \sum_{0 \leq j \leq \ell} \|h_{ij}(\bar X)(X_0+\cdots+X_n)^{\ell-j} \|   }
\\[6mm]
& \leq & \displaystyle{ \sum_{0 \leq i \leq m}  \sum_{0 \leq j \leq \ell} \|h_{ij}(\bar X) \| }  
\\[6mm] 
& \leq  & (m + 1)(\ell + 1)\| h \|_\bullet.
\end{array}
$$

We use now the bound for P\'olya's Theorem from \cite[Theorem 1] {PR}.
Take $N \in \N$ given by 
$$
N = \left\lfloor \frac{(m+1)(\ell+1)\ell(\ell-1)\|h\|_{\bullet}}{f^{\bullet}} - \ell \right\rfloor + 1.
$$
Then for each $(y,z) \in C$ we have that 
$
H\big(X_0,{\bar{X}},y,z\big)\left(X_0 + X_1 + \cdots + X_n \right)^N \in \R[X_0, \bar X]$
is a homogeneous polynomial such that all its coefficients are positive. 
More precisely, suppose we write
\begin{equation}\label{desig_8}
H\big(X_0,{\bar{X}},Y, Z\big)\left(X_0 + X_1 + \cdots + X_n \right)^N
= 
\sum_{\substack{\alpha=(\alpha_0,\bar \alpha) \in \N_0^{n+1} \\ |\alpha|=N+\ell}} 
b_{\alpha}(Y, Z) X_0^{\alpha_0} {\bar{ X}}^{\bar \alpha}
\in \R[X_0, \bar X, Y, Z]
\end{equation}
with $b_\alpha \in \R[Y, Z]$ homogeneous of degree $m$.
The conclusion is that for every $\alpha \in \N_0^{n+1}$ with $|\alpha|=N+\ell$, 
the polynomial $b_\alpha$ is positive in $C$, and therefore, 
since it is a homogenous polynomial, $b_\alpha$
is non-negative in $\R^2$. 

Before going on, we bound $N + \ell$ using (\ref{cota_e}) and (\ref{cota_h}) as follows.

\begin{equation}\label{eq:cota_Nmasell}
 \begin{split}
  N+ \ell & \leq \frac{(m+1)(\ell +1)\ell (\ell-1)\|h\|_{\bullet}}{f^{\bullet}}+1 \\
 & \leq \frac{(m+1)\ell^3\|h\|_{\bullet}}{f^{\bullet}}+1  \\
 & \leq 
c_5^3c_6(m+1)2^{\frac{m}2}
\left(\frac{ \| f \|_{\bullet}(m+1)d(d+1)}{{f^{\bullet}}}\right)^{4(c_1+1)}
{\rm e}^{c_7 \left( \frac{\| f \|_{\bullet}(m+1)d(d+1)}{f^{\bullet}} \right)^{c_1+1} }
  +1  \\
 & \leq 
c_8(m+1)2^{\frac{m}2}
\left(\frac{ \| f \|_{\bullet}(m+1)d(d+1)}{{f^{\bullet}}}\right)^{4(c_1+1)}
{\rm e}^{c_7 \left( \frac{\| f \|_{\bullet}(m+1)d(d+1)}{f^{\bullet}} \right)^{c_1+1} }
  \\
 \end{split}
\end{equation}
with $c_8 = c_5^3c_6 + 1$.

Now we substitute $X_0=1-X_1- \cdots -X_n$ and $Z =1$ in (\ref{desig_8})
and we obtain
\begin{equation} \label{eq:almost_finish}
f =  
\lambda\big(Y^2 + 1\big)^{\frac{m}2}\sum_{1 \le i \le s}g_i(g_i-1)^{2k} 
+
\sum_{\substack{\alpha=(\alpha_0,\bar \alpha) \in \N_0^{n+1} \\ |\alpha|=N+\ell}}  
b_{\alpha}(Y, 1) (1 - X_1 - \dots - X_n)^{\alpha_0} {\bar{ X}}^{\bar \alpha} \in \R[\bar X, Y].
\end{equation}
From (\ref{eq:almost_finish}) we want to conclude that $f \in M_{\R[\bar X, Y]}(g_1 \dots, g_s)$
and to find the positive constant $c$ such that the degree bound holds.

The first term on the right hand side 
of (\ref{eq:almost_finish}) clearly
belongs to $M_{\R[\bar X, Y]}(g_1 \dots, g_s)$. Moreover, 
for $1 \le i \le s$, 
\begin{equation}\label{eq:cota_1er_term}
\deg \big(Y^2 + 1\big)^{\frac{m}2}g_i(g_i-1)^{2k} \le m + (2k+1)\deg g_i.
\end{equation}

Now we focus on the second term on the right hand side 
of (\ref{eq:almost_finish}), which is itself a sum. Take a fixed 
$\alpha \in \N_0^{n+1}$ with $|\alpha|=N+\ell$.

Since 
$b_\alpha(Y, 1)$ is non-negative in $\R$, $b_\alpha(Y, 1) \in \sum \R[Y]^2$. 
Moreover, we can write $b_\alpha(Y, 1)$ as a sum of squares with 
the degree of each square bounded by $m$ (see \cite[Proposition 1.2.1]{Marshall_book}).

Also, take $v(\alpha) =  (v_0, \bar v) \in \{0, 1\}^{n+1}$ 
such that $\alpha_i \equiv v_i \, (\mbox{mod} \ 2)$ for $0 \le i \le n$. 
Denoting $g_0 = 1 \in \R[\bar X]$, by Remark \ref{obs:exis}, we have
$$
(1-X_1- \cdots -X_n)^{v_0} \bar{ X}^{\bar v} =  \sum_{0 \le i \le s}
\sigma_{v(\alpha) i} g_i, 
$$
with $\sigma_{v(\alpha)i} \in \sum \R[\bar X]^2$ for $0 \le i \le s$, and 
then 
$$
(1 - X_1 - \dots - X_n)^{\alpha_0} {\bar{ X}}^{\bar \alpha}
= 
(1 - X_1 - \dots - X_n)^{\alpha_0-v_0} {\bar{ X}}^{\bar \alpha - \bar v}
\sum_{0 \le i \le s}
\sigma_{v(\alpha) i} g_i
$$
belongs to $M(g_1, \dots, g_s)$
since $(1 - X_1 - \dots - X_n)^{\alpha_0-v_0} {\bar{ X}}^{\bar \alpha - \bar v}
\in \R[\bar X]^2$. 

We conclude in this way that each term in the sum belongs to $M_{\R[\bar X, Y]}(g_1, \dots, g_s)$.
In addition, 
for $0 \le i \le s$ we have
\begin{equation}\label{eq:cota_2do_term}
\deg 
b_{\alpha}(Y, 1) (1 - X_1 - \dots - X_n)^{\alpha_0-v_0} 
{\bar{ X}}^{\bar \alpha - \bar v} \sigma_{v(\alpha)i}g_i 
\le m + N + \ell + c_9
\end{equation}
with $c_9 = \max \{ \deg \sigma_{vi}g_i \ |
\ v \in \{0, 1\}^{n+1},  0 \le i \le s\}$.

To finish the proof, we only need to bound simultaneously the right hand side
of (\ref{eq:cota_1er_term}) and 
(\ref{eq:cota_2do_term}).

On the one hand, using (\ref{eq:cota_k}), 
\begin{equation*}
\begin{split}
m + (2k+1)\max_{1 \le i \le s} \deg g_i & \le 
m + \left(2 c_4 \left( \frac{\|\bar f \|_{\bullet}(m+1)d(d+1)}{f^{\bullet}} \right)^{c_1+1} + 1\right) 
\max_{1 \le i \le s} \deg g_i \\
& \le
c_{10}(m +1)  \left(\frac{\|\bar f \|_{\bullet}(m+1)d^2}{f^{\bullet}} \right)^{c_1+1}  
\end{split}
\end{equation*}
with $c_{10} = (2c_4 + 1)2^{c_1 + 1}\displaystyle{\max_{1 \le i \le s} \deg g_i}$, since 
$d \ge 1$. 

On the other hand, using (\ref{eq:cota_Nmasell}), 
\begin{equation*}
\begin{split}
m + N + \ell + c_9 &
\le 
m + 
c_8(m+1)2^{\frac{m}2}
\left(\frac{ \| f \|_{\bullet}(m+1)d(d+1)}{{f^{\bullet}}}\right)^{4(c_1+1)}
{\rm e}^{c_7 \left( \frac{\|\bar f \|_{\bullet}(m+1)d(d+1)}{f^{\bullet}} \right)^{c_1+1} } + c_9 \\
&
\le
c_{11}(m+1)2^{\frac{m}2}
\left(\frac{ \| f \|_{\bullet}(m+1)d^2}{{f^{\bullet}}}\right)^{4(c_1+1)}
{\rm e}^{c_{12} \left( \frac{\|f \|_{\bullet}(m+1)d^2}{f^{\bullet}} \right)^{c_1+1} }  \\
\end{split}
\end{equation*}
with $c_{11} = (1 + c_8 + c_9)2^{4(c_1+1)}$ and $c_{12} = c_72^{c_1 + 1}$, again since 
$d \ge 1$. 

Finally, we define $c$ as the positive constant obtained applying  Lemma \ref{lem:ctes} to 
the $6$-uple $(c_{10},$ $c_1+1,$ $c_{11}, 4(c_1 + 1), c_{12}, c_1 + 1)$.
\end{proof}

Before going to the proof of our main result, we include a technical lemma with a bound. 

\begin{lemma}\label{lem:techn_bound}
Let $j, d \in \N_0$ with $j \le d$. 
Then
$$2^j\binom{d+1}{j+1} \le 3^d.$$
\end{lemma}

\begin{proof}{Proof:} We proceed by induction on $d$. 
For $d = 0$ we check the inequality by hand. Now suppose $d \ge 1$.
If $j = 0$ or $j = d$ again we check the inequality by hand. If $1 \le j \le d-1$ then 
$$
2^j\binom{d+1}{j+1} 
= 2\cdot 2^{j-1}\binom{d}{j} + 2^j\binom{d}{j+1} \le 2\cdot 3^{d-1} + 3^{d-1} = 3^d.
$$
\end{proof}

We are ready to prove Theorem \ref{thm:main}.

\begin{proof}{Proof of Theorem \ref{thm:main}:}
We consider the affine change of variables $\ell:\R^n \to \R^n$ given by
$$
\ell(X_1, \dots, X_n) = \left(\frac{X_1 + 1}{2n}, \dots, 
\frac{X_n + 1}{2n}\right).
$$
For $0 \le i \le s$, we take 
$\tilde{g}_i(\bar X)=g_i(\ell^{-1}(\bar X)) \in \R[\bar X]$ and we define
$$
\widetilde{S}= \{\bar x \in \R^n \ | \ \tilde{g}_1(\bar x) \geq 0, \dots,  
\tilde{g}_s(\bar x) \ge 0 \}.
$$
It is easy to see that
$$
\emptyset \ne \widetilde{S} = \ell(S) \subseteq \widetilde{\Delta}_n^{\circ}.
$$
Moreover, since $M(g_1, \dots, g_n)$ is archimedean, 
$M(\tilde g_1, \dots, \tilde g_s)$ is also archimedean. To see this, following
the proof of \cite[Proposition 5.2.3]{Marshall_book}, if 
$$
N - X_1^2 - \dots - X_n^2  \in M(g_1, \dots, g_n),
$$
then
$$
\frac{N}{2n^2} + \frac{1}{2n}- \left(\frac{X_1+1}{2n}\right)^2 - \dots - \left(\frac{X_n+1}{2n}\right)^2 =
$$
$$
= \frac1{2n^2}\left(N - X_1^2 - \dots - X_n^2\right) +
\frac1{4n^2}\left( (X_1 - 1)^2 + \dots + (X_n - 1)^2 \right) \in M(g_1, \dots, g_n)
$$
and composing with $\ell^{-1}$
we have
$$
\frac{N}{2n^2} + \frac{1}{2n} - X_1^2 - \dots - X_n^2  \in M(\tilde g_1, \dots, 
\tilde g_n).
$$

Let 
$f \in \R[\bar X, Y]$ be as in the statement of Theorem \ref{thm:main}
and let $\tilde{f}(\bar{X},Y)=f(\ell^{-1}(\bar{X}),Y) \in \R[\bar X, Y]$.
It can be easily seen that
$\tilde f$ is positive on $\tilde S \times \R$, 
$\deg_{\bar X}\tilde f = \deg_{\bar X} f = d$, 
$\deg_{Y}\tilde f = \deg_{Y} f = m$, $\tilde f$ is fully $m$-ic on $\widetilde S$
and
$$
\min\{\bar{\tilde f}(\bar x, y, z) \,  | \, \bar x \in \widetilde S, (y, z) \in C \} 
= \min\{\bar{f}(\bar x, y, z) \,  | \, \bar x \in  S, (y, z) \in C \} = f^{\bullet} >0, 
$$ 

Now we want to bound $\| \tilde f \|_{\bullet}$.
So suppose 
$$
f= 
 \sum_{0\leq i \leq m} 
 \sum_{\substack{\alpha \in \N_0^{n} \\ |\alpha|\le d}} \binom{|\alpha|}{\alpha} a_{\alpha,i} {\bar X}^{\alpha} Y^i.
$$
Then 
\begin{equation*}
 \begin{split}
 {\tilde{f}}&= 
 \sum_{0\leq i \leq m} \sum_{\substack{\alpha \in \N_0^{n} \\ |\alpha|\le d}} \binom{|\alpha|}{\alpha} a_{\alpha,i} (2nX_1-1, \cdots ,2nX_n-1)^{\alpha}  Y^i \\[1mm]
 &= 
 \sum_{0\leq i \leq m} \sum_{\substack{\alpha \in \N_0^{n} \\ |\alpha|\le d}} 
 \binom{|\alpha|}{\alpha} a_{\alpha,i}  
 \sum_{\substack{\beta \in \N_0^{n} \\ \beta \preceq \alpha}}
 \binom{\alpha_1}{\beta_1} \cdots \binom{\alpha_n}{\beta_n}
  (2n)^{|\beta|}(-1)^{|\alpha|-|\beta|} \bar X^{\beta} Y^i  \\[1mm]
 &= 
 \sum_{0\leq i \leq m} \sum_{\substack{\beta \in \N_0^n \\ |\beta| \leq d}} (2n)^{|\beta|} 
 \sum_{\substack{ \alpha \in \N_0^n  \\ \beta \preceq \alpha, \, |\alpha| \leq d}}
 \binom{|\alpha|}{\alpha} \binom{\alpha_1}{\beta_1} \cdots \binom{\alpha_n}{\beta_n}
 a_{\alpha,i}  (-1)^{|\alpha|-|\beta|} \bar X^{\beta} Y^i \\[1mm]
 & = \sum_{0\leq i \leq m} \sum_{\substack{\beta \in \N_0^n \\ |\beta| \leq d}} 
 \binom{|\beta|}{\beta} (2n)^{|\beta|} \frac{1}{|\beta|!} 
 \sum_{\substack{ \alpha \in \N_0^n  \\ \beta \preceq \alpha, \, |\alpha| \leq d}}
 \frac{|\alpha|!}{ (\alpha_1-\beta_1)! \cdots (\alpha_n-\beta_n)!}a_{\alpha,i}  (-1)^{|\alpha|-|\beta|} \bar X^{\beta} Y^i
 \end{split}
\end{equation*}
For $0 \le i \le m$ and $\beta \in \N_0^n$ with $|\beta| \le d$ we define
$$
b_{\beta, i} = (2n)^{|\beta|} \frac{1}{|\beta|!} 
 \sum_{\substack{ \alpha \in \N_0^n  \\ \beta \preceq \alpha, \, |\alpha| \leq d}}
 \frac{|\alpha|!}{ (\alpha_1-\beta_1)! \cdots (\alpha_n-\beta_n)!}a_{\alpha,i}  (-1)^{|\alpha|-|\beta|}\in \R
$$ 
Then we have 
\begin{equation*}
 \begin{split}
|b_{\beta, i} | 
& \leq \| f\|_{\bullet}(2n)^{|\beta|} \frac{1}{|\beta|!}  
\sum_{\substack{ \alpha \in \N_0^n  \\ \beta \preceq \alpha, \, |\alpha| \leq d}}
\frac{|\alpha|!}{ (\alpha_1-\beta_1)! \cdots (\alpha_n-\beta_n)!} 
\\[1mm]
& = \| f\|_{\bullet}(2n)^{|\beta|} \frac{1}{|\beta|!}  \sum_{|\beta| \leq h \leq d} h! 
\sum_{\substack{ \alpha \in \N_0^n  \\ \beta \preceq \alpha, \, |\alpha| = h}}
\frac{1}{ (\alpha_1-\beta_1)! \cdots (\alpha_n-\beta_n)!} 
\\[1mm]
& = \| f\|_{\bullet}(2n)^{|\beta|} \frac{1}{|\beta|!}  \sum_{|\beta| \leq h \leq d}  \frac{h!}{(h-|\beta|)!}
\sum_{\substack{ \gamma \in \N_0^n \\ |\gamma|=h-|\beta|}}
\binom{|\gamma|}{\gamma}  
\\[1mm]
& =\| f\|_{\bullet}(2n)^{|\beta|}   \sum_{|\beta| \leq h \leq d}  \binom{h}{|\beta|}
n^{h-|\beta|} \\[1mm]
& \le  
\| f\|_{\bullet} 2^{|\beta|}n^{d} \sum_{|\beta| \leq h \leq d}  \binom{h}{|\beta|}
\\[1mm]
& =
\| f\|_{\bullet} 2^{|\beta|}n^{d} \binom{d+1}{|\beta| + 1}
\\[1mm]
&\leq \| f\|_{\bullet} (3n)^d
 \end{split}
\end{equation*}
using Lemma \ref{lem:techn_bound}. So the conclusion is $\| \tilde f \|_{\bullet}
\le \| f \|_{\bullet}(3n)^d$.

Finally, take $c$ as the positive constant from Proposition \ref{prop:main} applied to 
$\tilde g_1, \dots, \tilde g_s$. Therefore, 
$\tilde f$ can be written as 
$$
\tilde f = \tilde \sigma_0 + \tilde \sigma_1\tilde g_1 + \dots + \tilde\sigma_s \tilde g_s \in 
M_{\R[\bar X, Y]}(\tilde g_1, \dots, \tilde  g_s)
$$
with $\tilde\sigma_0, \tilde\sigma_1,  \dots, \tilde\sigma_s \in \sum \R[\bar X, Y]^2$ and 
$$
\deg(\tilde\sigma_0), \deg(\tilde\sigma_1 \tilde g_1),  \dots, \deg(\tilde\sigma_s\tilde g_s)
\le 
c (m+1)2^{\frac{m}{2}}
{\rm e}^{\left( \frac{\|  f \|_{\bullet}(m+1)d^2(3n)^d}{f^{\bullet}} \right)^{c}}
$$
and the desired representation for $f$ is simply obtained by composing with $\ell$. 
\end{proof}

\begin{remark}\label{obs:3^d}
If $n \ge 2$, 
the factor 
$3^d$ in the exponentiation
in 
the bound from Theorem \ref{thm:main}
can be hidden in the constant $c$ as follows. If $d = 0$, as explained before, 
the degree bound holds for any $c \ge 1$. If $d \ge 1$, then, 
using Lemma \ref{lem:norma_div_min},
$$
\frac{\|  f \|_{\bullet}(m+1)d^2(3n)^d}{f^{\bullet}} \le
\frac{\|  f \|_{\bullet}(m+1)d(d+1)(3n)^d}{f^{\bullet}} 
\le 
$$
$$
\le 
\left(\frac{\|  f \|_{\bullet}(m+1)d(d+1)n^d}{f^{\bullet}}\right)^3  \le 
8 \left(\frac{\|  f \|_{\bullet}(m+1)d^2n^d}{f^{\bullet}}\right)^3.
$$
So we replace $c$ by the result of applying Lemma \ref{lem:ctes} to 
the $6$-uple $(1, 0, c, 0, 8^c, 3c)$ and we are done.  
\end{remark}

\textbf{Acknowledgements:} We are thankful to the reviewers for their
helpful
suggestions and remarks.

\end{document}